\documentclass[a4paper,11pt]{article}
\usepackage
[a4paper,left=1in,right=1in,
top=1in,bottom=1in]{geometry}
\usepackage{amssymb,amsmath,amsfonts,amsthm}

\newcommand{\ipref}{\ensuremath{\mathcal{IP}}}
\newcommand{\supp}{\ensuremath{\text{supp}}}
\newcommand{\upr}{\ensuremath{a}}
\newcommand{\ilen}{\ensuremath{\ell_{\text{inv}}}}
\newcommand{\rank}{\ensuremath{\text{rank}}}

\theoremstyle{plain}
\newtheorem{thm}{Theorem}[section]
\newtheorem{lemma}[thm]{Lemma}        
\newtheorem{prop}[thm]{Proposition}
\newtheorem{defn}[thm]{Definition}
\newtheorem{property}[thm]{Property}

\newtheorem{conj}[thm]{Conjecture}

\begin{document}
 \title{\textbf{A  note on involution prefixes in Coxeter groups}}
    \date{}  \author{Sarah B. Hart\thanks{School of Computing and Mathematical Sciences,
    		Birkbeck, University of London, Malet Street, London WC1E 7HX, United Kingdom, s.hart@bbk.ac.uk}$\;$ and Peter J. Rowley\thanks{School of Mathematics, The University of Manchester, Oxford Road, Manchester M13 9PL, United Kingdom, peter.j.rowley@manchester.ac.uk}}
  \maketitle

\begin{abstract} Let $(W, R)$ be a Coxeter system and let $w \in W$. We say that $u$ is a {\em prefix} of $w$ if there is a reduced expression for $u$ that can be extended to one for $w$. That is, $w = uv$ for some $v$ in $W$ such that $\ell(w) = \ell(u) + \ell(v)$. We say that $w$ has the {\em ancestor property} if the set of prefixes of $w$ contains a unique involution of maximal length. In this paper we show that all Coxeter elements of finitely generated Coxeter groups have the ancestor property, and hence a canonical expression as a product of involutions. We  conjecture that the property in fact holds for all non-identity elements of finite Coxeter groups.  
	(MSC2000: 20F55)
\end{abstract}

\section{Introduction}

Involutions play an important role in the study of Coxeter groups. By definition, any Coxeter group $W$ is generated by a distinguished set $R$ of involutions, known as the simple reflections. Involutions in Coxeter groups are well understood: conjugates of simple reflections are called reflections, and any
involution can be expressed as a canonical product of reflections
(see \cite{deodhar} and \cite{springer}). The involution conjugacy classes can be determined from the Coxeter graph with a simple algorithm due to Richardson
\cite{richardson}. Every element of a Coxeter group $W$ can be written as a product of involutions (just take any expression for it as a product of simple reflections). But if we work with the full set of involutions, expressions with fewer terms may be possible. For example, if $W$ is finite, every element of $W$ is a product of at most two involutions \cite[Lemma~5 and Theorem~6]{carter}. These two extremes behave markedly differently with respect to the length function. Recall that in a Coxeter system $(W, R)$, the {\em length} of an element $w$ of $W$, written $\ell(w)$, is the minimal length of any expression for $w$ as a product of simple reflections (by convention, the identity element has length zero). Such minimal length expressions are called reduced expressions. Any reduced expression $w = r_1r_2\cdots r_n$ (where $r_i \in R$ for each $i$) by definition has the property that the length of $w$ is the sum of the lengths of the involutions in this product. Expressions for $w$ as a product of two involutions do not typically have this additive length property. That is, it cannot be guaranteed that involutions $u_1$ and $u_2$ can be chosen such that $w = u_1u_2$ and $\ell(w) = \ell(u_1) + \ell(u_2)$. (The minimum attainable value of $\ell(u_1) + \ell(u_2) - \ell(w)$ is called the \textit{excess} of $w$. Among other things, it has been shown (see \cite{invprod}) that every element of a finite Coxeter group is conjugate to an element with excess zero.) There may be many expressions for $w$ as a product of two or more involutions which do have the additive length property. A natural question is whether one can write each $w$ as a product of involutions $u_1$, $\ldots$, $u_k$ such that $\ell(w) = \ell(u_1) + \cdots + \ell(u_k)$, in some canonical way -- that is, unique subject to a simple criterion. In this paper we present a conjecture that indeed this can be done, and prove it in the special case of Coxeter elements.     

Let $W$ be a Coxeter group with $R$ the set of simple reflections, and let $w\in W$. We say that $u$ is a {\em prefix} of $w$ if there are $r_1, \ldots, r_n \in R$, and $0\leq k\leq n$ such that $r_1\cdots r_k$ is a reduced expression for $u$ and $r_1\cdots r_n$ is a reduced expression for $w$. We denote by $\ipref(w)$ the set of all prefixes of $w$ which are involutions. That is,  $$\ipref(w) = \{u \in W: u {\text{ is a prefix of $w$ and $o(u) = 2$}}\},$$
and we call the elements of $\ipref(w)$ the {\em involution prefixes} of $w$. We further define the set $A(w)$ of {\em ancestors} of $w$ to be the set of involution prefixes of $w$ whose length is maximal. That is, $$A(w) = \bigg\{u \in \ipref(w): \ell(u) = \max_{v\in \ipref(w)}\{\ell(v)\}\bigg\}.$$
\begin{property} \label{property1} Let $W$ be a Coxeter group and $w$ a non-identity element of $W$. We say that $w$ has the {\em ancestor property} if $|A(w)| = 1$. In this case, we call the unique element of $A(w)$ the {\em ancestor} of $w$, and denote it $\upr(w)$. We say that $W$ has the ancestor property if every non-identity element of $W$ has this property.
	\end{property}

\begin{conj}[The Ancestor Conjecture]
	\label{conj1} Every finite Coxeter group has the ancestor property. 
\end{conj}

There is a connection between Conjecture~\ref{conj1} and the weak (left) Bruhat order. Conjecture~\ref{conj1} is equivalent to saying that every interval $[1,w]$ in the weak Bruhat order (with $w\in W\setminus \{1\}$) contains a unique highest involution.
 
Suppose $W$ has the ancestor property. Then every non-identity element of $W$ has a unique decomposition as a product of ancestors: if $w$ is an involution, then $\upr(w) = w$; otherwise, inductively, $w = \upr_1\cdots \upr_k$ for some $k > 1$, where for $1\leq i < k$ we have $\upr_i = \upr(\upr_{i}\cdots \upr_{k})$. We call this expression for $w$ the {\em ancestor decomposition} of $w$, and define the {\em involution length} of $w$, or $\ilen(w)$, to be the number $k$ of involutions making up the expression, with the convention that $\ilen(1) = 0$.

\begin{conj}\label{conj2}
Suppose Conjecture~\ref{conj1} holds, and let $W$ be a finite Coxeter group. For each $w \in W$, the involution length of $w$ is at most the rank of $W$. 
\end{conj}

In Section \ref{sec2} we show that if $W$ is a finitely generated Coxeter group, then every Coxeter element of $W$ has the ancestor property. We note that the length of any Coxeter element $w$ of $W$ equals the rank of $W$. It is clear that $\ilen(w) \leq \ell(w)$, so it follows immediately that in this special case, the involution length of any Coxeter element is at most the rank of $W$. We further show in Lemma \ref{chromatic} that the minimum involution length possible for a Coxeter element of $W$ is the chromatic number of the Coxeter graph of $W$. At the other extreme, Proposition \ref{path} shows that the maximum possible involution length is the maximum path length of the Coxeter graph. 

In Section \ref{sec3} we summarise calculations verifying the conjectures for some finite Coxeter groups, and discuss what happens in infinite Coxeter groups.

\section{Coxeter elements}\label{sec2}

Throughout this section let $(W,R)$ be a Coxeter system group of finite rank $n$. 
The {\em (left) descent set} of $w$ in $W$ is the set $D(w) = \{r\in R: \ell(rw) < \ell(w)\}$.
Equivalently, $D(w)$ is the set of simple reflections which can appear at the start of a reduced expression for $w$, so that $D(w) = \ipref(w)\cap R$.
A {\em Coxeter element} of $W$ is a product of all the simple reflections in some order, with each $r\in R$ appearing exactly once. All Coxeter elements have the same order, known as the {\em Coxeter number} of the group, and when $W$ is finite they are all conjugate \cite[\S 3.16]{humphreys}.  
Recall that the {\em support} of an element $w$ of $W$, denoted $\supp(w)$ is the set of simple reflections that appear in a reduced expression for $w$. (This is independent of the choice of reduced expression, by \cite[\S 5.10]{humphreys}.) 

\begin{lemma}\label{descentscommute} 
Let $w$ be a Coxeter element of $W$. Then all elements of $D(w)$ commute. Hence, $\prod_{r\in D(w)}r$ is both well-defined and an involution. 
\end{lemma}

\begin{proof}
	It is well known that all reduced expressions for a given element of
	a Coxeter group are obtainable from each other by repeated
	applications of braid relations (that is, interchanging a string
	$rsrs\cdots$ of length $m$ for a string $srsr\ldots$ of length $m$
	for some $r, s \in R$ with $(rs)^m = 1$). In a Coxeter element all
	reduced expressions contain exactly one occurrence of each $r \in
	R$. Hence the only relations which could be used are instances where
	$rs = sr$. Let $r_1\cdots r_n$ be any reduced expression for $w$, and suppose $r_i \in D(w)$. Then there is a reduced expression for $w$ that begins with $r_i$, and it is obtainable from $r_1\cdots r_n$ by braid relations of the form $rs = sr$. Hence, $r_i$ commutes with $r_1$ (and indeed with all elements to its left in any reduced expression for $w$). If $s$ is any other element of $D(w)$, then there is at least one reduced expression for $w$ that begins with $s$, and thus $r$ commutes with $s$. This means that $D(w)$ is a collection of mutually commuting simple reflections; their product in any order gives the same element of $W$, and so $\prod_{r\in D(w)}r$ is indeed both well-defined and an involution. 
\end{proof}

\begin{prop} \label{one} Let $W$ be a finitely generated Coxeter group. Each
	Coxeter element $w$ of $W$ has the ancestor property, and $\upr(w) = \prod_{r\in D(w)}r$. 
\end{prop}

\begin{proof} Write $x = \prod_{r\in D(w)} r$. Then $x$ is an involution prefix of $w$. Suppose $y$ is another involution prefix of $w$. Then $y$ is a Coxeter element of a standard parabolic subgroup of $W$. Since $y$ is an involution, the corresponding Coxeter graph of this subgroup is edgeless, and hence $y$ is a product of mutually commuting simple reflections. Thus, $D(y) = \supp(y)$. Since $D(y)\subseteq D(w)$, we have that $\ell(y)\leq \ell(x)$, with equality if and only if $y=x$. Therefore $x = \upr(w)$, as required.
\end{proof}

By Proposition \ref{one}, every Coxeter element of $W$ has a unique ancestor decomposition, and hence the involution length is well defined for Coxeter elements of finitely generated Coxeter groups. Since the ancestor decomposition is a reduced expression, it is clear that for a Coxeter element $w$ we have $\ilen(w) \leq \ell(w) = \rank(W)$. Thus, Conjecture \ref{conj2} is trivially true for such elements. But we are able to give considerably more detail about which involution lengths are possible for Coxeter elements of a given group, and also a criterion to determine the involution length of a given Coxeter element. 

The Coxeter graph of $W$ is denoted $\Gamma(W)$, or just $\Gamma$ if the choice of group is clear.

\begin{lemma}\label{chromatic}
	The  minimum involution length attainable by a Coxeter element of $W$ is the chromatic number of its Coxeter graph.
\end{lemma}

\begin{proof}
	Let $w$ be a Coxeter element of $W$, of involution length $k$, and let $\upr_1\cdots \upr_k$ be its ancestor decomposition. By Proposition~\ref{one}, $\supp(\upr_i)$ is an independent vertex set of $\Gamma$. Thus the ancestor decomposition induces a $k$-colouring of $\Gamma$, and $\chi(\Gamma) \leq \ilen(w)$. For the reverse implication we proceed by induction on $\chi(\Gamma)$. If $\chi(\Gamma) = 1$, then every Coxeter element is an involution, so has involution length 1. Suppose $\chi(\Gamma)=k>1$. Let $R = R_1\cup R_2\cup \cdots\cup R_k$ be a partition of the vertex set $R$ of $\Gamma$ into $k$ independent sets. By moving additional independent vertices to $R_1$ if necessary, we may assume $R_1$ is maximal by inclusion. Let $R_1 = \{r_1, \ldots, r_m\}$ and let $w$ be any Coxeter element of the form $r_1\cdots r_m\widehat w$, where $\widehat w$ is an arbitrary Coxeter element of the standard parabolic subgroup $\widehat{W}$ generated by $R\setminus R_1$, whose Coxeter graph $\widehat\Gamma$ has chromatic number $k-1$. Clearly $R_1 \subseteq D(w)$. Moreover, since $R_1$ is maximal by inclusion, no element $r$ of $R\setminus R_1$ commutes with every element of $R_1$. Therefore, $D(w) = R_1$ and hence $\upr(w) = r_1\cdots r_m$. Inductively, $\ilen(\widehat w) = \chi(\widehat \Gamma) = k-1$. Hence $\ilen(w) = k = \chi(\Gamma)$, as required.  
\end{proof}

\begin{defn}
	Let $w$ be a Coxeter element of $W$. We call $r_{1}, \ldots, r_{m}$ a $w$-path if $r_{1}, \ldots, r_{m}$ is a path in $\Gamma$ and there is a reduced expression for $w$ in which $r_i$ appears to the left of $r_{i+1}$ for each $1\leq i < m$.
		The {\em path length of $w$} is the maximum length of any $w$-path. 	
\end{defn}

Suppose that $r_{1}, \ldots, r_{m}$ is a $w$-path, so that there is a reduced expression for $w$ in which $r_i$ appears to the left of $r_{i+1}$ for $1\leq i<m$. The only braid relations that can be used in Coxeter elements are of the form $rs=sr$, and thus the order of appearance of any two simple reflections in any expression for $w$ can only change if they commute. But $r_i$ is joined by an edge to $r_{i+1}$ in $\Gamma$, so this does not happen, meaning that $r_{i}$ must appear to the left of $r_{i+1}$ in every reduced expression for $w$. We may therefore, without ambiguity, refer to $w$-paths without having to specify a particular reduced expression. 

\begin{prop}\label{path} 
Let $w$ be a Coxeter element of a finitely generated Coxeter group $W$. Then $\ilen(w) = {\text{ path length}}(w)$. Hence, the maximum involution length attainable by a Coxeter element of $W$ is the maximum path length of its Coxeter graph. 
\end{prop}

\begin{proof} We proceed by induction on $\ilen(w)$. If $\ilen(w) = 1$, then $w$ is an involution, which implies $\Gamma$ is edge-free and so the path length of $w$ is clearly 1. Hence $\ilen(w) = {\text{ path length}}(w)$. Suppose $\ilen(w) > 1$. Let the ancestor decomposition of $w$ be $\upr_1\cdots \upr_k$. 
Consider a maximal $w$-path $r_1$, \ldots, $r_m$ (that is, a $w$-path to which no further terms can be added to make a longer path). Then $r_{1}$ commutes with all terms to its left in $w$. Hence, $r_{1} \in D(w) = \supp(\upr(w))$.
Simple reflections corresponding to adjacent vertices of $\Gamma$ do  not commute, and therefore cannot lie in the same $\upr_j$ for any $j$. Thus, $r_2$, $\ldots$, $r_m$ lie in involutions strictly to the right of $r_1$, and so are not contained in $D(w)$. Thus $\supp(\upr(w))$ consists precisely of the initial vertices of maximal $w$-paths. Setting $w' = \upr_2\cdots \upr_k$, inductively $$\ilen(w) = \ilen(w')+1 =  {\text{ path length}}(w')+1 = {\text{ path length}}(w),$$ and in fact for each $i$ we see that $\upr_i$ consists precisely of the $i^{\text{th}}$ vertices of maximal $w$-paths. Hence, by induction, $\ilen(w) = {\text{ path length}}(w)$. It follows immediately that the maximum involution length of a Coxeter element of $W$ is the maximum path length of its Coxeter graph.  
\end{proof}

By Lemma \ref{chromatic}, there exist Coxeter elements of involution length 2 if and only if the Coxeter graph is bipartite (which is equivalent to it containing no odd cycles). In particular, as any tree or forest is bipartite, every finite Coxeter group except $\mathrm{A}_1^n$ contains Coxeter elements of involution length 2. At the other extreme, we have already observed that the involution length of any Coxeter element of $W$ cannot exceed the rank of $W$. By Proposition \ref{path}, a Coxeter group of rank $n$ will contain Coxeter elements of involution length $n$ if and only if its Coxeter graph has maximum path length $n$. The only trees with this property are paths. Therefore the only finite Coxeter groups $W$ which contain Coxeter elements of involution length $\rank(W)$ are those whose Coxeter graph is a path -- that is, types $\mathrm{A}_n$, $\mathrm{B}_n$, $\mathrm{F}_4$, $\mathrm{H}_3$, $\mathrm{H}_4$, and $\mathrm{I}_2(m)$. 

\section{Conclusion}\label{sec3}

Conjectures \ref{conj1} and \ref{conj2} are clearly true for dihedral groups. For all other irreducible Coxeter groups of order less than 100,000 (namely types $\mathrm{A}_n$ for $n\leq 7$, $\mathrm{B}_n$ and $\mathrm{D}_n$ for $n\leq 6$, $\mathrm{F}_4$, $\mathrm{E}_6$, $\mathrm{H}_3$, and $\mathrm{H}_4$), Conjectures \ref{conj1} and \ref{conj2} have been verified using Magma \cite{magma}. It follows immediately from the definition that if Conjecture \ref{conj1} holds for finite irreducible Coxeter groups, then it holds for finite reducible Coxeter groups (and the same is true for Conjecture \ref{conj2}). Therefore, it is a corollary of the computer calculations just referred to that both conjectures hold for all Coxeter groups of order less than 100,000.

For infinite Coxeter groups, Conjecture \ref{conj2} is false in general, even if the group has the ancestor property. For example, in the Coxeter group $G_n := \langle r_1, \ldots, r_n | r_1^2 = \cdots = r_n^2 = 1\rangle$, 
there is a unique reduced expression for every element. Therefore, any $w$ in $G_n$ has exactly one prefix of length $i$, for $0 \leq i \leq \ell(w)$. Hence $G_n$ has the ancestor property. But Conjecture~\ref{conj2} is false: for every positive integer $k$, the element $(r_1\cdots r_n)^k$ of $G_n$ has involution length $nk$. 

Note that we could make all of these definitions for suffixes rather than prefixes,  resulting in different decompositions and potentially different involution
lengths. For example in $\mathrm{A}_6$ with the usual labelling of
fundamental reflections, the element $r_6r_3r_2r_1r_4r_5$ has
the `suffix ancestor' decomposition $(r_3)(r_2r_4r_6)(r_1r_5)$ and
ancestor decomposition $(r_3r_6)(r_2r_4)(r_1r_5)$. We suspect that the  involution length of any element always equals its `suffix' involution length, but have not tested this extensively enough to make a firm conjecture. However, if Conjectures \ref{conj1} and \ref{conj2} hold for prefixes, then they also hold for suffixes, because the involution suffixes of any element $w$ are the involution prefixes of $w^{-1}$.

{}

\end{document}